\documentclass{amsart}
\usepackage{amsmath, amsthm, enumerate, tikz}
\usetikzlibrary{decorations.markings}

\newtheorem{theorem}{Theorem}[section]
\newtheorem{corollary}[theorem]{Corollary}
\newtheorem{lemma}[theorem]{Lemma}

\newtheorem{observation}[theorem]{Observation}

\theoremstyle{definition}
\newtheorem{definition}[theorem]{Definition}
\theoremstyle{remark}
\theoremstyle{definition}
\newtheorem{example}[theorem]{Example}

\numberwithin{equation}{section}

\definecolor{helena}{rgb}{.2,.8,.4}

\begin{document}

\allowdisplaybreaks

\title[An extension of the Hermite-Biehler theorem]{An extension of the Hermite-Biehler theorem with application to polynomials with one positive root}

%    Only \author and \address are required; other information is
%    optional.  Remove any unused author tags.

%    author one information
% \author[short version for running head]{name for top of paper}
\author{Richard Ellard}
\address{
	Richard Ellard,
	School of Mathematics and Statistics,
	University College Dublin,
	Belfield, Dublin 4, Ireland
}
\email{richardellard@gmail.com}
\thanks{The authors' work was supported by Science Foundation Ireland under Grant 11/RFP.1/MTH/3157.}

%    author two information
\author{Helena \v{S}migoc}
\address{
	Helena \v{S}migoc,
	School of Mathematics and Statistics,
	University College Dublin,
	Belfield, Dublin 4, Ireland
}
\email{helena.smigoc@ucd.ie}

%    \subjclass is required.


\date{January 2017}

%    Abstract is required.
\begin{abstract}
	If a real polynomial $f(x)=p(x^2)+xq(x^2)$ is Hurwitz stable (every root if $f$ lies in the open left half-plane), then the Hermite-Biehler Theorem says that the polynomials $p(-x^2)$ and $q(-x^2)$ have interlacing real roots. We extend this result to general polynomials by giving a lower bound on the number of real roots of $p(-x^2)$ and $q(-x^2)$ and showing that these real roots interlace. This bound depends on the number of roots of $f$ which lie in the left half plane. Another classical result in the theory of polynomials is Descartes' Rule of Signs, which bounds the number of positive roots of a polynomial in terms of the number of sign changes in its coefficients. We use our extension of the Hermite-Biehler Theorem to give an inverse rule of signs for polynomials with one positive root.
\end{abstract}

\maketitle

\section{Introduction}

Recall that a real polynomial $f$ is called \emph{(Hurwitz) stable} if every root of $f$ lies in the open left half-plane. Determining the stability of real polynomials is of fundamental importance in the study of dynamical systems and as such, several equivalent characterisations have been given. One such characterisation is the \emph{Hermite-Biehler Theorem} \cite{Hermite1856, Biehler1879}, a proof of which can also be found in \cite{Holtz2003}. The Hermite-Biehler Theorem has been instrumental in the study of the ``robust parametric stability problem'', that is, the problem of guaranteeing that stability is preserved by real coefficient perturbations (see \cite{Kharitonov1978, Bhattacharyya1987}).

\begin{theorem}[\bf Hermite-Biehler Theorem]
	Let
	\[
		f(x):=a_0x^n+a_1x^{n-1}+\cdots+a_n
	\]
	be a real polynomial and write $f(x)=p(x^2)+xq(x^2)$, where $p(x^2)$ and $xq(x^2)$ are the components of $f(x)$ made up by the even and odd powers of $x$, respectively. Let $x_{e1},x_{e2},\ldots$ denote the distinct nonnegative real roots of $p(-x^2)$ and let $x_{o1},x_{o2},\ldots$ denote the distinct nonnegative real roots of $q(-x^2)$, where both sequences are arranged in ascending order. Then $f$ is stable if and only if the following conditions hold:
	\begin{enumerate}
		\item[\textup{(i)}] all of the roots of $p(-x^2)$ and $q(-x^2)$ are real and distinct;
		\item[\textup{(ii)}] $a_0$ and $a_1$ have the same sign;
		\item[\textup{(ii)}] $0<x_{e1}<x_{o1}<x_{e2}<x_{o2}<\cdots$.
	\end{enumerate}
\end{theorem}

The Hermite-Biehler theorem says that, if $f(x)=p(x^2)+xq(x^2)$ is stable, then the polynomials $p(-x^2)$ and $q(-x^2)$ have real, interlacing roots. In Section \ref{sec:dRealRoots}, we will extend the Hermite-Biehler Theorem by showing that, even if $f$ is not stable (suppose $f$ has $n_-$ roots in the left half-plane and $n_+$ roots in the right), then it is still possible to give a lower bound on the number of real roots of $p(-x^2)$ and $q(-x^2)$. This bound is given in terms of the quantity $|n_--n_+|$. Furthermore, we show that these real roots interlace.

Another classical result in the theory of polynomials is Descartes' Rule of Signs. We say that a real polynomial
\[
	f(x)=a_0x^n+a_1x^{n-1}+a_2x^{n-2}+\cdots+a_n, \hspace{6mm} a_0\neq0
\]
\emph{has $k$ sign changes} if $k$ sign changes occur between consecutive nonzero elements of the sequence $a_0,a_1,\ldots,a_n$. Descartes' Rule of Signs states that the number of positive roots of $f$ is either equal to $k$, or is less than $k$ by an even number. Descartes' rule gives the exact number of positive roots in only two cases:
\begin{enumerate}[(i)]
	\item $f$ has no sign changes, in which case, $f$ has no positive roots, or
	\item $f$ has precisely one sign change, in which case, $f$ has precisely one positive root.
\end{enumerate}

Conversely to (i), if every root of $f$ has real part less than or equal to zero, then $f$ has no sign changes. To see this, we need only observe that, if the roots of $f$ are labeled $-\eta_1,-\eta_2,\ldots,-\eta_s,-\alpha_1\pm i\beta_1,-\alpha_2\pm i\beta_2,\ldots,-\alpha_m\pm i\beta_m$, where $\eta_j,\alpha_j,\beta_j\geq0$ and $s+2m=n$, then the polynomial
\[
	\frac{1}{a_0}f(x)=\prod_{j=1}^s(x+\eta_j)\prod_{j=1}^m\left( (x+\alpha_j)^2+\beta_j^2 \right)
\]
has nonnegative coefficients, and consequently, every nonzero coefficient of $f$ has the same sign.

In general, the converse of (ii) is not true; however, in Section \ref{sec:StrongerLSLemma}, we will use our extension of the Hermite-Biehler Theorem to prove that, if $f$ has at most one root with positive real part, then the sequences $a_0,a_2,a_4,\ldots$ and $a_1,a_3,a_5,\ldots$ each feature at most one sign change.

Polynomials with one positive root (in particular, inverse rules of signs for such polynomials) are of interest in a number of areas, such as in polynomial real root isolation, i.e. the process of finding a collection of intervals of the real line such that each interval contains precisely one real root and each real root is contained in some interval. Modern real root isolation algorithms typically use a version of \emph{Vincent's Theorem} \cite{Vincent1836}, the proof of which depends on some kind of inverse rule of sign for polynomials with one positive root. For example, the proof of Vincent's Theorem given by Alesina and Galuzzi \cite{AlesinaGaluzzi2000} uses a special case of a theorem of Obreschkoff \cite{Obreschkoff1966}, which we state below:

\begin{theorem}{\bf\cite{Obreschkoff1966}}\label{thm:Obreschkoff}
	If a real polynomial $f$ of degree $n$ has a simple positive root $r$ and all other roots lie in the wedge
	\begin{equation}\label{eq:Sr3}
		S_{\sqrt{3}}:=\{-\alpha+i\beta : \alpha>0,\, |\beta|\leq\sqrt{3}\alpha\},
	\end{equation}
	then $f$ has precisely one sign change.
\end{theorem}

%\begin{theorem}[Vincent's Theorem, see \cite{Vincent1836, AlesinaGaluzzi2000}]\label{thm:Vincent}
%	Let $f$ be a real polynomial of degree $n$ with simple roots. Then there exists a positive quantity $\delta$ such that for every pair of positive real numbers $a$ and $b$ with $a<b$ and $b-a<\delta$, the polynomial
%	\[
%		\phi_{a,b}(x):=(1+x)^nf\left( \frac{a+bx}{1+x} \right)
%	\]
%	has either zero or one sign change. The latter case occurs if and only if $f$ has a single root in $(a,b)$.
%\end{theorem}

Polynomials with one positive root also arise in problems that consider the sign patterns of matrices (in particular, companion and related matrices). One such problem is the \emph{Nonnegative Inverse Eigenvalue Problem}, or NIEP. This is the (still open) problem of characterising those lists of complex numbers which are \emph{realisable} as the spectrum of some (entrywise) nonnegative matrix. Polynomials with one positive root are of particular importance in the NIEP, and as such, the NIEP has already motivated several results on the coefficients of polynomials of this type. In this context, the polynomial $f$ represents the characteristic polynomial of the realising matrix and its one positive root represents the Perron eigenvalue of the realising matrix.

One of the earliest results in the NIEP was given by Sule\v{i}manova \cite{Suleimanova1949} when she proved the following:

\begin{theorem}{\bf\cite{Suleimanova1949}}
	Let $\sigma:=(\rho,\lambda_2,\lambda_3,\ldots,\lambda_n)$, where $\rho\geq0$ and $\lambda_i\leq0:$ $i=2,3,\ldots,n$. Then $\sigma$ is the spectrum of a nonnegative matrix if and only if
	\[
		\rho+\lambda_2+\lambda_3+\cdots+\lambda_n\geq0.
	\]
\end{theorem}

Perhaps the most elegant proof of Sule\v{i}manova's result is due to Perfect \cite{Perfect1953}, who showed that, under the assumptions of the theorem, every coefficient of the polynomial
\[
	f(x)=(x-\rho)\prod_{i=2}^n(x-\lambda_i),
\]
apart from the leading coefficient, is nonpositive, and hence, the companion matrix of $f$ is nonnegative (note that, since Sule\v{i}manova's hypotheses guarantee the coefficient of $x^{n-1}$ in $f$ is negative, the same result follows immediately from Theorem \ref{thm:Obreschkoff}).

Later, Laffey and \v{S}migoc \cite{LaffeySmigoc} generalised Sule\v{i}manova's theorem to complex lists with one positive element and $n-1$ elements with real part less than or equal to zero:

\begin{theorem}{\bf\cite{LaffeySmigoc}}\label{thm:LS}
  Let $\rho\geq0$ and let $\lambda_2,\lambda_3,\ldots,\lambda_n$ be complex numbers such that
$\mathrm{Re}\,\lambda_i\leq0$ for all $i=2,3,\ldots,n$. Then the list
$\sigma:=(\rho,\lambda_2,\lambda_3,\ldots,\lambda_n)$ is the spectrum of a nonnegative
matrix if and only if the following conditions hold:
  \begin{enumerate}[(i)]
    \item[\textup{(i)}]
    $\sigma$ is self-conjugate;
    \item[\textup{(ii)}]
    $\rho+\lambda_2+\lambda_3+\cdots+\lambda_n\geq0$;
    \item[\textup{(iii)}]
    $(\rho+\lambda_2+\lambda_3+\cdots+\lambda_n)^2\leq n(\rho^2+\lambda_2^2+\lambda_3^2+\cdots+\lambda_n^2)$.
  \end{enumerate}
  Furthermore, when the above conditions are satisified, $\sigma$ may be realised by a matrix of the form
$C+\alpha I_n$, where $C$ is a nonnegative companion matrix with trace zero and $\alpha$ is a
nonnegative scalar.
\end{theorem}

The crucial ingredient in Laffey and \v{S}migoc's result was the following lemma (also proved by the authors):

\begin{lemma}\textup{\bf\cite{LaffeySmigoc}}\label{lem:LS}
	Let $(\lambda_2,\lambda_3,\ldots,\lambda_n)$ be a self-conjugate list of complex numbers with nonpositive real parts, let $\rho\geq0$ and let
	\[
		f(x):=(x-\rho)\prod_{i=2}^n(x-\lambda_i)=x^n+a_1x^{n-1}+a_2x^{n-2}+\cdots+a_n.
	\]
	If $a_1,a_2\leq0$, then $a_i\leq0:$ $i=3,4,\ldots,n$.
\end{lemma}

Although Lemma \ref{lem:LS} was motivated by matrix theory, it is, fundamentally, a result on the coefficients of real polynomials. We generalise this result in Section \ref{sec:StrongerLSLemma}.

\section{The Cauchy index of a rational function}

\begin{definition}
	Let $f(x)$ be a real rational function and let $\theta,\phi\in\mathbb{R}\cup\{-\infty,\infty\}$, with $\theta<\phi$. The \emph{Cauchy index} of $f(x)$ between the limits $\theta$ and $\phi$---written $I_\theta^\phi f(x)$---is defined as the number of times $f(x)$ jumps from $-\infty$ to $\infty$, minus the number of times $f(x)$ jumps from $\infty$ to $-\infty$, as $x$ moves from $\theta$ to $\phi$.
\end{definition}

\begin{example}
	If
	\[
		f(x)=\frac{1}{(x+1)(x-1)},
	\]
	then $I_{-\infty}^0f(x)=-1$, $I_0^\infty f(x)=1$ and $I_{-\infty}^\infty f(x)=0$.
\end{example}

We introduce some additional notation: if $f(x)$ is a complex-valued function and $C$ is a contour in the complex plane, let $\Delta_Cf(x)$ denote the total increase in $\mathrm{arg}\,f(x)$ as $x$ traverses the contour $C$. If $C$ is the line segment from $\theta$ to $\phi$, then we write $\Delta_\theta^\phi f(x)$.

The following result (and its proof) essentially appears in \cite[Chapter 15, \S3]{Gantmacher}. The proof is included for completeness.

\begin{theorem}[See \cite{Gantmacher}]\label{thm:CauchyIndex}
	Let $f(x):=P(x)+iQ(x)$, where
	\[
		P(x):=x^n+a_1x^{n-1}+a_2x^{n-2}+\cdots+a_n
	\]
	and
	\[
		Q(x):=b_1x^{n-1}+b_2x^{n-2}+\cdots+b_n
	\]
	are real polynomials. Suppose $f$ has $n_+$ roots with positive imaginary part, $n_-$ roots with negative imaginary part and $n_0$ real roots ($n_++n_-+n_0=n$). Then
	\[
		I_{-\infty}^\infty\frac{Q(x)}{P(x)}=n_--n_+.
	\]
\end{theorem}
\begin{proof}
	We first consider the case when $n_0=0$. Define the closed contour $C=C_1+C_2$ (shown in Figure \ref{fig:Contour}), where $C_1$ is the line segment from $-R$ to $R$ and $C_2$ is the semicircle
	\[
		x(t)=Re^{it} \hspace{2mm} : \hspace{10mm} 0\leq t\leq\pi.
	\]
	Assume $R$ is large enough so that all of the roots of $f$ with positive imaginary part lie within the region enclosed by $C$.
	
\tikzset{->-/.style={decoration={
  markings,
  mark=at position #1 with {\arrow{>}}},postaction={decorate}}}

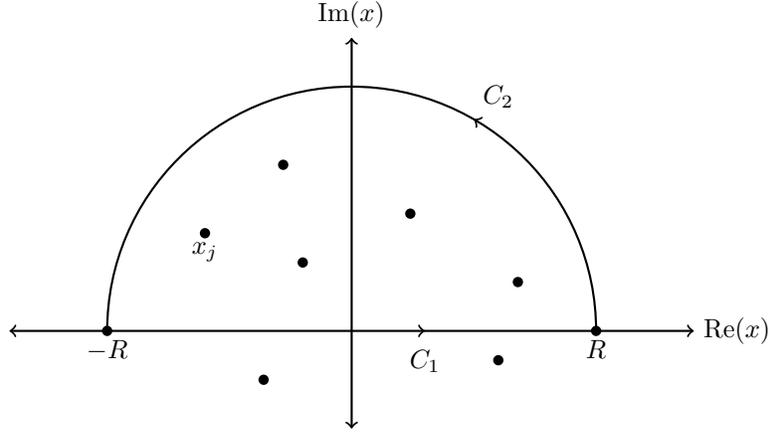
\begin{figure}[h!]
	\centering
	\begin{tikzpicture}[scale=1.3]
		\draw [<->,thick] (0,-1) -- (0,3) node [above] {$\mathrm{Im}(x)$};
		\draw [->-=.3,thick] (0,0) -- (2.5,0) node {};
		\draw [->,thick] (2.5,0) -- (3.5,0) node [right] {$\mathrm{Re}(x)$};
		\draw (1.5,2.4) node {$C_2$};
		\draw (.3*2.5,-0.1) node [below] {$C_1$};
		\draw [->,thick] (0,0) -- (-3.5,0) node [right] {};
		\draw [->-=1/3,thick] (2.5,0) arc (0:180:2.5);
		\fill (-2.5,0) circle [radius=1.5pt];
		\draw (-2.5,0) node [below] {$-R$};
		\fill (2.5,0) circle [radius=1.5pt];
		\draw (2.5,0) node [below] {$R$};
		%roots
		\fill (-1.5,1) circle [radius=1.5pt];
		\draw (-1.5,1) node [below] {$x_j$};
		\fill (-0.7,1.7) circle [radius=1.5pt];
		\fill (-0.5,0.7) circle [radius=1.5pt];
		\fill (0.6,1.2) circle [radius=1.5pt];
		\fill (1.7,0.5) circle [radius=1.5pt];
		\fill (-0.9,-0.5) circle [radius=1.5pt];
		\fill (1.5,-0.3) circle [radius=1.5pt];
	\end{tikzpicture}
	\caption{Contours $C_1$ and $C_2$}
	\label{fig:Contour}
\end{figure}
		
	Denote the roots of $f$ by $x_1,x_2,\ldots,x_n$. For each $j=1,2,\ldots,n$, if $\mathrm{Re}(x_j)>0$, then $\Delta_C(x-x_j)=2\pi$. Otherwise, $\Delta_C(x-x_j)=0$. Therefore
	\[
		\Delta_Cf(x)=\Delta_C\left((a_0+ib_0)\prod_{j=1}^n(x-x_j)\right)=\sum_{j=1}^n\Delta_C(x-x_j)=2n_+\pi.
	\]
	Similarly,
	\[
		\lim_{R\rightarrow\infty}\Delta_{C_2}f(x)=n\pi.
	\]
	Hence
	\begin{equation}\label{eq:IProof1}
		\Delta_{-\infty}^\infty f(x)=(2n_+-n)\pi;
	\end{equation}
	however, since
	\[
		\mathrm{arg}\:f(x)=\tan^{-1}\frac{Q(x)}{P(x)}
	\]
	and
	\[
		\lim_{x\rightarrow\pm\infty}\frac{Q(x)}{P(x)}=0,
	\]
	it follows that
	\begin{equation}\label{eq:IProof2}
		\frac{1}{\pi}\Delta_{-\infty}^\infty f(x)=-I_{-\infty}^\infty\frac{Q(x)}{P(x)}.
	\end{equation}
	Combining (\ref{eq:IProof1}) and (\ref{eq:IProof2}) gives
	\[
		I_{-\infty}^\infty\frac{Q(x)}{P(x)}=n-2n_+=n_--n_+,
	\]
	as required.
	
	Now consider the case when $n_0>0$. Let us label the real roots of $f$ as $\eta_1,\eta_2,\ldots,$ $\eta_{n_0}$.
	
	Writing
	\begin{align*}
		f(x)&=\left( \prod_{j=1}^{n_0}(x-\eta_j) \right)\tilde{f}(x),\\
		\tilde{f}(x)&=\tilde{P}(x)+i\tilde{Q}(x),
	\end{align*}
	we note that the polynomial $\tilde{f}$ has $n_+$ roots with positive imaginary part, $n_-$ roots with negative imaginary part and no real roots. Hence, from the above,
	\[
		I_{-\infty}^\infty\frac{\tilde{Q}(x)}{\tilde{P}(x)}=n_--n_+.
	\]
	We note, however, that
	\begin{align*}
		P(x)&=\left( \prod_{j=1}^{n_0}(x-\eta_j) \right)\tilde{P}(x),\\
		Q(x)&=\left( \prod_{j=1}^{n_0}(x-\eta_j) \right)\tilde{Q}(x)
	\end{align*}
	and for all $j=1,2,\ldots,n_0$,
	\[
		\lim_{x\rightarrow \eta_j}\frac{Q(x)}{P(x)}=\lim_{x\rightarrow \eta_j}\frac{\tilde{Q}(x)}{\tilde{P}(x)}.
	\]
	Therefore
	\[
		I_{-\infty}^\infty\frac{Q(x)}{P(x)}=I_{-\infty}^\infty\frac{\tilde{Q}(x)}{\tilde{P}(x)}.
	\]
\end{proof}

%In the literature, Theorem \ref{thm:CauchyIndex} is often used to compute the number of roots of a real polynomial $f$ which have positive real part, and in this scenario, the Cauchy index may be calculated by Routh's algorithm.\footnote{See, for example, \cite[Chapter 15]{Gantmacher}.} In particular, whether $f$ is stable may be determined in this way.

\section{An extension of the Hermite-Biehler Theorem}\label{sec:dRealRoots}

In this section, we consider an arbitrary real polynomial $f(x)=p(x^2)+xq(x^2)$, with $n_-$ roots in the left half-plane and $n_+$ roots in the right half-plane. We extend the Hermite-Biehler Theorem by giving a lower bound on the number of real roots of $p(-x^2)$ and $q(-x^2)$ in terms of $|n_--n_+|$ and showing that these real roots interlace.

\begin{definition}
	Let $\mathcal{X}$ and $\mathcal{Z}$ be sequences of real numbers. We say $\mathcal{X}$ and $\mathcal{Z}$ \emph{interlace} if the following two conditions hold:
	\begin{enumerate}[(i)]
		\item if $x_i$ and $x_j$ are two distinct elements of $\mathcal{X}$ with $x_i<x_j$, then there exists an element $z_k$ of $\mathcal{Z}$ such that $x_i\leq z_k\leq x_j$ (and vice versa);
		\item if $x_i$ appears in $\mathcal{X}$ with multiplicity $m$, then $x_i$ appears in $\mathcal{Z}$ with multiplicity at least $m-1$ (and vice versa).
	\end{enumerate}
	We say $\mathcal{X}$ and $\mathcal{Z}$ \emph{strictly interlace} if every element of $\mathcal{X}$ and $\mathcal{Z}$ occurs with multiplicity 1, $\mathcal{X}$ and $\mathcal{Z}$ have no element in common and whenever $x_i$ and $x_j$ are two distinct elements of $\mathcal{X}$ with $x_i<x_j$, there exists an element $z_k$ of $\mathcal{Z}$ such that $x_i<z_k<x_j$ (and vice versa).
\end{definition}

Before considering the real polynomial $f(x)=p(x^2)+xq(x^2)$, it is easier (and more general) to first consider the complex polynomial $f(x):=P(x)+iQ(x)$.

\begin{theorem}\label{thm:Real/ImaginaryPartsVariant}
	Consider the polynomial
	\begin{equation*}%\label{eq:fInRIPVariant}
		f(x):=P(x)+iQ(x),
	\end{equation*}
	where
	\begin{align*}
		P(x)&:=x^n+a_1x^{n-1}+a_2x^{n-2}+\cdots+a_n, \\
		Q(x)&:=b_1x^{n-1}+b_2x^{n-2}+\cdots+b_n
	\end{align*}
	and the $a_i$ and $b_i$ are real. Suppose $f$ has $n_+$ roots with positive imaginary part, $n_-$ roots with negative imaginary part and $n_0<n$ real roots $(n_++n_-+n_0=n)$. If $d:=n-2\min\{n_+,n_-\}$, then (counting multiplicities) there exist at least $d$ real roots of $P$ (say $\mu_1,\mu_2,\ldots,\mu_d$) and at least $d-1$ real roots of $Q$ (say $\nu_1,\nu_2,\ldots,\nu_{d-1}$) such that
	\begin{equation}\label{eq:VariantInterlacing}
		\mu_1\leq\nu_1\leq\mu_2\leq\nu_2\leq\cdots\leq\nu_{d-1}\leq\mu_d.
	\end{equation}
	If $n_0=0$, then the inequalities in (\ref{eq:VariantInterlacing}) may be assumed to be strict.
\end{theorem}

%\begin{theorem}\label{thm:Real/ImaginaryPartsVariant}
%	Consider the polynomial
%	\begin{equation*}%\label{eq:fInRIPVariant}
%		f(x):=x^n+(a_1+ib_1)x^{n-1}+\cdots+(a_n+ib_n),
%	\end{equation*}
%	where the $a_i$ and $b_i$ are real. Suppose $f$ has $n_+$ roots with positive imaginary part, $n_-$ roots with negative imaginary part and $n_0<n$ real roots $(n_++n_-+n_0=n)$. Let
%	\begin{align*}
%		P(x)&:=x^n+a_1x^{n-1}+a_2x^{n-2}+\cdots+a_n, \\
%		Q(x)&:=b_1x^{n-1}+b_2x^{n-2}+\cdots+b_n
%	\end{align*}
%	and $d:=n-2\min(n_+,n_-)$. Then (counting multiplicities) there exist at least $d$ real roots of $P$ (say $\mu_1,\mu_2,\ldots,\mu_d$) and at least $d-1$ real roots of $Q$ (say $\nu_1,\nu_2,\ldots,\nu_{d-1}$) such that
%	\begin{equation}\label{eq:VariantInterlacing}
%		\mu_1\leq\nu_1\leq\mu_2\leq\nu_2\leq\cdots\leq\nu_{d-1}\leq\mu_d.
%	\end{equation}
%	If $n_0=0$, then the inequalities in (\ref{eq:VariantInterlacing}) may be assumed to be strict.
%\end{theorem}
\begin{proof}
	As in the proof of Theorem \ref{thm:CauchyIndex}, we first consider the case when $n_0=0$. In this case, $P$ and $Q$ can have no real root in common, since if $x_0$ were a real root of both $P$ and $Q$, then $x_0$ would also be a real root of $f$. Suppose also that $n_->n_+$.
	
	Let $p_1<p_2<\cdots<p_s$ be the points on the real line at which $Q(x)/P(x)$ jumps from $-\infty$ to $\infty$ and let $q_1<q_2<\cdots<q_{s'}$ be the points on the real line at which $Q(x)/P(x)$ jumps from $\infty$ to $-\infty$. Clearly, the $p_i$ and $q_i$ are roots of $P$. Suppose they are arranged as follows:
	\begin{gather*}
		\cdots<p_{k_j}< p_{k_j+1}<\cdots< p_{k_{j+1}-1}\\
		<q_{l_j}< q_{l_j+1}<\cdots< q_{l_{j+1}-1}\\
		<p_{k_{j+1}}< p_{k_{j+1}+1}<\cdots< p_{k_{j+2}-1}<\cdots.
	\end{gather*}
	
	Now consider the interval $R:=(p_{k_j+r-1},p_{k_j+r})$, where $1\leq r\leq k_{j+1}-k_j-1$. By definition of the $p_i$,
	\[
		\lim_{x\rightarrow p_{r-1+k_j}^+}\frac{Q(x)}{P(x)}=\infty, \hspace{10mm}
		\lim_{x\rightarrow p_{r+k_j}^-}\frac{Q(x)}{P(x)}=-\infty.
	\]
	Furthermore, although $Q(x)/P(x)$ may have discontinuities in $R$ (at points where $P$ has a root of even multiplicity), $Q(x)/P(x)$ does not change sign at these discontinuities. Hence $Q(x)/P(x)$ has a root, say $w_{jr}$, in $R$. Obviously, $w_{jr}$ is also a root of $Q$.
	
	Let us now consider the sequence
	\begin{multline*}
		\mathcal{T}:=(\,\ldots,p_{k_j},w_{j1}, p_{k_j+1},w_{j2},\ldots,p_{k_{j+1}-1},\\
		p_{k_{j+1}},w_{j+1,1}, p_{k_{j+1}+1},w_{j+1,2},\ldots,p_{k_{j+2}-1},\ldots\,).
	\end{multline*}
	This sequence consists of strictly interlacing roots of $P$ and $Q$, apart from certain pairs of adjacent roots of $P$ of the form $(p_{k_{j+1}-1},p_{k_{j+1}})$. Hence, we form a new sequence $\mathcal{T}'$ from $\mathcal{T}$ by deleting either $p_{k_{j+1}-1}$ or $p_{k_{j+1}}$ for each $j$. Since $\mathcal{T}'$ is a strictly interlacing sequence of real roots of $P$ and $Q$, whose first and last entries are roots of $P$, it is sufficient to check that $\mathcal{T}'$ is sufficiently long.
	
	Let $h$ be the number of subsequences $(q_{l_j}< q_{l_j+1}<\cdots< q_{l_{j+1}-1})$ which lie between $p_1$ and $p_s$. We note that $\mathcal{T}$ has length $2s-h-1$. Since $\mathcal{T}'$ was formed by deleting $h$ elements from $\mathcal{T}$, it follows that $\mathcal{T}'$ has length
	\[
		2(s-h)-1\geq2(s-s')-1=2I_{-\infty}^\infty\frac{Q(x)}{P(x)}-1.
	\]
	By Theorem \ref{thm:CauchyIndex}, it follows that $\mathcal{T}'$ has at least
	\[
		2(n_--n_+)-1=2(n-2n_+)-1=2d-1
	\]
	elements, as required.
	
	We have yet to consider $n_+\geq n_-$ or $n_0>0$. If $n_0=0$ and $n_+=n_-$, then the statement says nothing; hence we may ignore this case. If $n_0=0$ and $n_+>n_-$, then the proof is analogous to the above.
	
	Finally, suppose $n_0>0$. Let us label the real roots of $f$ as $\eta_1,\eta_2,\ldots,$ $\eta_{n_0}$. Writing
	\[
		f(x)=\left( \prod_{j=1}^{n_0}(x-\eta_j) \right)\left( \tilde{P}(x)+i\tilde{Q}(x) \right),
	\]
	we note that the polynomial $\tilde{P}(x)+i\tilde{Q}(x)$ has $n_+$ roots with positive imaginary part, $n_-$ roots with negative imaginary part and no real roots. Hence, from the above, there exist $d-n_0$ real roots of $\tilde{P}$ (say $\mu_1,\mu_2,\ldots,\mu_{d-n_0}$) and $d-n_0-1$ real roots of $\tilde{Q}$ (say $\nu_1,\nu_2,$ $\ldots,\nu_{d-n_0-1}$) such that
	\[
		\mu_1<\nu_1<\mu_2<\nu_2<\cdots<\nu_{d-n_0-1}<\mu_{d-n_0}.
	\]
	All that remains is to note that the sequences
	\[
		(\mu_1,\mu_2,\ldots,\mu_{d-n_0},\eta_1,\eta_2,\ldots,\eta_{n_0})
	\]
	and
	\[
		(\nu_1,\nu_2,\ldots,\nu_{d-n_0-1},\eta_1,\eta_2,\ldots,\eta_{n_0})
	\]
	interlace (though not strictly).
\end{proof}

As a consequence of Theorem \ref{thm:Real/ImaginaryPartsVariant}, we obtain the following extension of the Hermite-Biehler theorem:

\begin{corollary}\label{cor:Even/OddVariant}
	Consider the real polynomial
	\[
		f(x):=x^n+a_1x^{n-1}+a_2x^{n-2}+\cdots+a_n.
	\]
	Suppose $f$ has $n_+$ roots with positive real part, $n_-$ roots with negative real part and $n_0<n$ purely imaginary roots $(n_++n_-+n_0=n)$. Let
	\begin{align}
		P(x)&:=x^n-a_2x^{n-2}+a_4x^{n-4}-\cdots,\notag \\
		Q(x)&:=a_1x^{n-1}-a_3x^{n-3}+a_5x^{n-5}-\cdots \label{eq:QPolynomial}
	\end{align}
	and $d:=n-2\min(n_+,n_-)$. Then (counting multiplicities) there exist at least $d$ real roots of $P$ (say $\mu_1,\mu_2,\ldots,\mu_d$) and at least $d-1$ real roots of $Q$ (say $\nu_1,\nu_2,\ldots,\nu_{d-1}$) such that
	\begin{equation}\label{eq:VariantInterlacing'}
		\mu_1\leq\nu_1\leq\mu_2\leq\nu_2\leq\cdots\leq\nu_{d-1}\leq\mu_d.
	\end{equation}
	If $n_0=0$, then the inequalities in (\ref{eq:VariantInterlacing'}) may be assumed to be strict.
\end{corollary}
\begin{proof}
	The real parts of the roots of $f$ correspond to the imaginary parts of the roots of the polynomial
	\[
		g(x):=i^nf(-ix)=x^n+ia_1x^{n-1}-a_2x^{n-2}-ia_3x^{n-3}+\cdots.
	\]
	The result follows from Theorem \ref{thm:Real/ImaginaryPartsVariant}.
\end{proof}

Note that the bounds given for the number of real roots of $P$ and $Q$ in Corollary \ref{cor:Even/OddVariant} may or may not be achieved, as illustrated by the following two examples:

%\begin{example}
%	The polynomial
%	\[
%		f(x)=x^5-i x^4-3 x^3-4 x+i
%	\]
%	has $n_+=4$ roots with positive imaginary part and $n_-=1$ root with negative imaginary part. Taking the real part of $f(x)$ gives the real polynomial $P(x)=x^5-3 x^3-4 x$ with roots $-2,0,2,i,-i$ and taking the imaginary part gives the polynomial $Q(x)=-x^4+1$ with roots $-1,1,i,-i$. In this example, the bound given in Theorem \ref{thm:Real/ImaginaryPartsVariant} on the number of real roots of $P$ and $Q$ is achieved.
%\end{example}

\begin{example}
	The polynomial
	\[
		f(x):=x^5-x^4+3 x^3-4 x+1
	\]
	has $n_+=4$ roots with positive real part and $n_-=1$ root with negative real part, so that, in the notation of Corollary \ref{cor:Even/OddVariant}, $d=3$. The polynomial $P(x):=x^5-3 x^3-4 x$ has roots $-2,0,2,i,-i$ and the polynomial $Q(x)=-x^4+1$ has roots $-1,1,i,-i$. Hence, in this example, the bounds given in Corollary \ref{cor:Even/OddVariant} on the numbers of real roots of $P$ and $Q$ is achieved.
\end{example}

\begin{example}
	Consider the polynomial
	\[
		f(x):=x^4+2 x^3+23 x^2+94 x+130,
	\]
	with roots $1\pm5i,-2\pm i$. In the notation of Corollary \ref{cor:Even/OddVariant}, $n_+=n_-=2$ and $d=0$. Hence, the corollary does not guarantee the existence of any real roots of the polynomials
	\[
		P(x):=x^4-23 x^2+130
	\]
	or
	\[
		Q(x):=2 x^3-94 x;
	\]
	however, $P$ has roots $-\sqrt{13},-\sqrt{10},\sqrt{10},\sqrt{13}$ and $Q$ has roots $-\sqrt{47},0,\sqrt{47}$.
\end{example}

It turns out that, under certain circumstances, we can infer the existence of an additional two real roots of the polynomial $Q$ given in (\ref{eq:QPolynomial}). We will use these additional roots in the next section.

\begin{observation}\label{obv:TwoMoreRoots}
	Assume the hypotheses and conclusion of Theorem \ref{thm:Real/ImaginaryPartsVariant} (alternatively Corollary \ref{cor:Even/OddVariant}).
	\begin{enumerate}
		\item[\textup{(i)}] If $n_->n_+$ and $\lim_{x\rightarrow-\infty}(P(x)/Q(x))=\infty$, or alternatively if $n_-<n_+$ and $\lim_{x\rightarrow-\infty}(P(x)/Q(x))=-\infty$, then there exists an additional real root $\nu_0$ of $Q$ such that $\nu_0\leq\mu_1$.
		\item[\textup{(ii)}] If $n_->n_+$ and $\lim_{x\rightarrow\infty}(P(x)/Q(x))=-\infty$, or alternatively if $n_-<n_+$ and $\lim_{x\rightarrow\infty}(P(x)/Q(x))=\infty$, then there exists an additional real root $\nu_d$ of $Q$ such that $\nu_d\geq\mu_d$.
	\end{enumerate}
	If $n_0=0$, then $\nu_0<\mu_1$ and $\nu_d>\mu_d$.
\end{observation}
\begin{proof}
	Assume the hypotheses and conclusion of Theorem \ref{thm:Real/ImaginaryPartsVariant} (those of Corollary \ref{cor:Even/OddVariant} are equivalent). First suppose $n_->n_+$ and
	\begin{equation}\label{eq:LimitInObvProof}
		\lim_{x\rightarrow-\infty}\frac{P(x)}{Q(x)}=\infty.
	\end{equation}
	In the proof of Theorem \ref{thm:Real/ImaginaryPartsVariant}, the first element $p_1$ of $\mathcal{T}'$ was chosen such that
	\[
		\lim_{x\rightarrow p_1^-}\frac{Q(x)}{P(x)}=-\infty.
	\]
	Hence, in this case, (\ref{eq:LimitInObvProof}) implies the existence of an additional real root $w_0$ of $Q(x)/P(x)$ such that $w_0<p_1$. It follows that there exists an additional real root $\nu_0$ of $Q$ such that $\nu_0\leq\mu_1$.
	
	The remaining cases are dealt with similarly.
\end{proof}

\section{Polynomials with one positive root}\label{sec:StrongerLSLemma}

Using our extension of the Hermite-Biehler theorem, it will now be possible give an inverse rule of signs for real polynomials with one positive root. Later (in Theorem \ref{thm:LSVariantReformulated}), we will show how this rule can be somewhat simplified, under some minor additional assumptions.

\begin{theorem}\label{thm:LSVariant}
	Consider the real polynomial
	\[
		f(x):=x^n+a_1x^{n-1}+a_2x^{n-2}+\cdots+a_n.
	\]
	Suppose $f$ has roots $r,x_2,x_3,\ldots,x_n$, where $r$ is real and $\mathrm{Re}(x_j)\leq0:$ $j=2,3,\ldots,n$. Then the sequence $a_1,a_2,\ldots,a_n$ satisfies the following conditions:
	\begin{enumerate}
		\item[\textup{(i)}] Let $t$ be the largest integer such that $a_{2t}\neq0$. Then either $a_{2j}>0$ for all $j=1,2,\ldots,t$, or there exists $s\in\{1,2,\ldots,t\}$ such that
		\begin{align*}
			a_{2j}&>0 \hspace{2mm} : \hspace{6mm} j=1,2,\ldots,s-1, \\
			a_{2s}&\leq0, \\
			a_{2j}&<0 \hspace{2mm} : \hspace{6mm} j=s+1,s+2,\ldots,t.
		\end{align*}
		\item[\textup{(ii)}] Let $t'$ be the largest integer such that $a_{2t'-1}\neq0$. Then either $a_{2j-1}>0$ for all $j=1,2,\ldots,t'$, or there exists $s'\in\{1,2,\ldots,t'\}$ such that
		\begin{align*}
			a_{2j-1}&>0 \hspace{2mm} : \hspace{6mm} j=1,2,\ldots,s'-1, \\
			a_{2s'-1}&\leq0, \\
			a_{2j-1}&<0 \hspace{2mm} : \hspace{6mm} j=s'+1,s'+2,\ldots,t'.
		\end{align*}
	\end{enumerate}
\end{theorem}
\begin{proof}
	First suppose $n$ is even and write $n=2m$. The polynomial
	\[
		f(x)=x^{2m}+a_1x^{2m-1}+a_2x^{2m-2}+\cdots+a_{2m}
	\]
	has at most one root with positive real part. Therefore, by Corollary \ref{cor:Even/OddVariant}, the polynomial
	\[
		x^{2m}-a_2x^{2m-2}+a_4x^{2m-4}-\cdots+(-1)^{m}a_{2m}
	\]
	has at least $2m-2$ real roots. It follows that the polynomial
	\[
		y^{m}-a_2y^{m-1}+a_4y^{m-2}-\cdots+(-1)^{m}a_{2m}
	\]
	has at least $m-1$ nonnegative roots. Let $t$ be the largest integer such that $a_{2t}\neq0$. Then the polynomial
	\[
		y^{t}-a_2y^{t-1}+a_4y^{t-2}-\cdots+(-1)^{t}a_{2t}
	\]
	has at least $t-1$ positive roots. Therefore, by Descartes' rule of signs, the number of sign changes which occur between consecutive nonzero terms of the sequence
	\[
		\mathcal{T}:=(1,-a_2,a_4,-a_6,\ldots,(-1)^{t}a_{2t})
	\]
	is at least $t-1$. In particular, since $\mathcal{T}$ contains $t+1$ elements, this implies at most one of the elements in $\mathcal{T}$ is zero. There are now three cases to consider:
	
	\underline{\textit{Case 1:}} If every element in $\mathcal{T}$ is nonzero and $\mathcal{T}$ has $t$ sign changes, then $a_{2j}>0$ for each $j=1,2,\ldots,t$.
	
	\underline{\textit{Case 2:}} If every element in $\mathcal{T}$ is nonzero and $\mathcal{T}$ has $t-1$ sign changes, then the sequence
	\[
		(1,a_2,a_4,\ldots,a_{2t})
	\]
	has precisely one sign change.
	
	\underline{\textit{Case 3:}} Suppose there exists $s\in\{1,2,\ldots,t\}$ such that $a_{2s}=0$. Then, removing $a_{2s}$ from $\mathcal{T}$, we obtain a sequence
	\[
		\mathcal{T}_0:=(1,-a_2,a_4,\ldots,(-1)^{s-1}a_{2s-2},(-1)^{s-1}a_{2s+2},\ldots,(-1)^{t}a_{2t})
	\]
	with $t$ elements (each nonzero) and $t-1$ sign changes. It follows that 
	\begin{align*}
		a_{2j}&>0 \hspace{2mm} : \hspace{6mm} j=1,2,\ldots,s-1, \\
		a_{2j}&<0 \hspace{2mm} : \hspace{6mm} j=s+1,s+2,\ldots,t.
	\end{align*}
	We have now shown that the sequence $a_2,a_4,\ldots$ satisfies condition (i).
		
	Similarly, by Corollary \ref{cor:Even/OddVariant}, the polynomial
	\begin{equation}\label{eq:T'Polynomial1}
		a_1x^{2m-1}-a_3x^{2m-3}+a_5x^{2m-5}-\cdots+(-1)^{m-1}a_{2m-1}x
	\end{equation}
	has at least $2m-3$ real roots, one of which is zero. It follows that the polynomial
	\[
		a_1y^{m-1}-a_3y^{m-2}+a_5y^{m-3}-\cdots+(-1)^{m-1}a_{2m-1}
	\]
	has at least $m-2$ nonnegative roots. Let $t'$ be the largest integer such that $a_{2t'-1}\neq0$. Then the polynomial
	\begin{equation}\label{eq:T'Polynomial2}
		a_1y^{t'-1}-a_3y^{t'-2}+a_5y^{t'-3}-\cdots+(-1)^{t'-1}a_{2t'-1}
	\end{equation}
	has at least $t'-2$ positive roots. Therefore, by Descartes' rule of signs, the number of sign changes which occur between consecutive nonzero terms of the sequence
	\[
		\mathcal{T}':=(a_1,-a_3,a_5,\ldots,(-1)^{t-1}a_{2t'-1})
	\]
	is at least $t'-2$. As above, this implies at most one of the elements in $\mathcal{T}'$ is zero.
	
	If $a_1>0$, then the sequences $\mathcal{T}$ and $\mathcal{T}'$ have the same properties. In this case, it follows from the above argument that the sequence $a_1,a_3,\ldots$ satisfies condition (ii).
	
	If $a_1<0$, then for $P(x)$ and $Q(x)$ defined as in (\ref{eq:QPolynomial}), we see that $$\lim_{x\rightarrow-\infty}(P(x)/Q(x))=\infty \text{ and }\lim_{x\rightarrow\infty}(P(x)/Q(x))=-\infty.$$ Hence, by Observation \ref{obv:TwoMoreRoots}, every root of (\ref{eq:T'Polynomial1}) is real. It follows that (\ref{eq:T'Polynomial2}) has $t'-1$ positive roots and $\mathcal{T}'$ has $t'-1$ sign changes. Therefore $a_{2j-1}<0$ for all $j=1,2,\ldots,t'$.
	
	Finally, if $a_1=0$, then consider the polynomial
	\begin{align*}
		f_\epsilon(x)&:=(x-r-\epsilon)\prod_{j=2}^n(x-x_j)\\
		&=x^n-\epsilon x^{n-1}+b_2x^{n-2}+b_3x^{n-3}+\cdots+b_n,
	\end{align*}
	where $\epsilon>0$. From the above, we see that $b_{2j-1}\leq0$: $j=2,3,\ldots,$ $\lceil n/2 \rceil$. Furthermore, since each $b_j$ depends continuously on $\epsilon$ and
	\[
		\lim_{\epsilon\rightarrow0}f_\epsilon(x)=f(x),
	\]
	it follows that $a_{2j-1}\leq0$: $j=2,3,\ldots,\lceil n/2 \rceil$. Since at most one of the elements in $\mathcal{T}'$ is zero, we conclude that $a_{2j-1}<0$ for all $j=2,3,\ldots,t'$. We have now shown that the sequence $a_1,a_3,\ldots$ satisfies condition (ii).
	
	The proof for odd $n$ is similar.
\end{proof}

With Corollary \ref{cor:Even/OddVariant} established, the proof of Theorem \ref{thm:LSVariant} is quite elementary. Furthermore, the proof generalises to polynomials which have more than one root with positive real part: by combining Corollary \ref{cor:Even/OddVariant} with Descartes' Rule of Signs, bounds can be given on the number of sign changes which occur in the even/odd coefficients.

The statement of Theorem \ref{thm:LSVariant} is somewhat complicated by the fact that the multiplicity of zero as a root of
\[
	x^n-a_2x^{n-2}+a_4x^{n-4}-\cdots
\]
may be different from the multiplicity of zero as a root of
\[
	a_1x^{n-1}-a_3x^{n-3}+a_5x^{n-5}-\cdots
\]
The following example illustrates this:

\begin{example}\label{ex:DegenerateExample}
	Let
	\[
		f(x):=(x-r)g(x)=x^{2m+2}+a_1x^{2m+1}+a_2x^{2m}+\cdots+a_{2m+2},
	\]
	where $r>0$ and
	\begin{equation}\label{eq:DegenerateCase}
		g(x):=(x+\mu)\prod_{j=1}^m(x^2+\beta_j^2) \hspace{2mm} : \hspace{6mm} \mu,\beta_1,\ldots,\beta_m>0.
	\end{equation}
	The constant term in $f$ is given by
	\[
		a_{2m+2}=-r\mu\beta_1^2\beta_2^2\cdots\beta_m^2<0.
	\]
	Hence, by Theorem \ref{thm:LSVariant}, the sequence $\mathcal{T}_e:=(1,a_2,a_4,\ldots,a_{2m+2})$ of even coefficients features precisely one sign change and at most one element of $\mathcal{T}_e$ vanishes.
	
	It is not difficult to verify that the odd coefficients of $f$ are given by
	\[
		a_{2k+1}=(\mu-r)e_k(\beta_1^2,\beta_2^2,\ldots,\beta_m^2) \hspace{2mm} : \hspace{6mm} k=0,1,\ldots,m,
	\]
	where $e_k$ denotes the $k$-th elementary symmetric function. Therefore, the sign of every odd coefficient is determined by the sign of $r-\mu$. In particular, if $r=\mu$, then every odd coefficient vanishes.
\end{example}

It turns out that Example \ref{ex:DegenerateExample} is essentially unique, in that, if $f$ is not of this form and $f(0)\neq0$, then $a_k\leq0$ implies $a_{k+2},a_{k+4},\ldots<0$. To establish this fact, we will require some inequalities from \cite{NewtonLikeInequalities}, which are closely related to \emph{Newton's Inequalities}:

\begin{theorem}{\bf\cite{NewtonLikeInequalities}}\label{thm:NewtonLikeInequalities}
	Let
	\[
		g(x):=\prod_{j=1}^n(x-x_j)=x^n+b_1x^{n-1}+b_2x^{n-2}+\cdots+b_n
	\]
	be a real polynomial, where $x_1,x_2,\ldots,x_n$ are complex numbers with nonpositive real parts. If $k$ and $l$ have different parity, $1\leq k<l\leq n-1$, then
	\begin{equation}\label{eq:NewtonLikeInequalities}
		b_kb_l\geq b_{k-1}b_{l+1}.
	\end{equation}
\end{theorem}

The case of equality in (\ref{eq:NewtonLikeInequalities}) is not explicitly considered in \cite{NewtonLikeInequalities}; however, by examining the proof, it is possible to characterise the equality case:

\begin{observation}\label{obv:EqualityCase}
	Assume the hypotheses of Theorem \ref{thm:NewtonLikeInequalities}. If $k$ is even and $l$ is odd, then equality occurs in (\ref{eq:NewtonLikeInequalities}) if and only if one of the following conditions holds:
	\begin{enumerate}
		\item[\textup{(i)}] zero is a root of $g$ of multiplicity at least $n-l+1$;
		\item[\textup{(ii)}] $\mathrm{Re}\,(x_j)=0$ for all $j$.
	\end{enumerate}
	If $k$ is odd and $l$ is even, then equality occurs in (\ref{eq:NewtonLikeInequalities}) if and only if (i) or (ii) holds, or $g$ is of the form (\ref{eq:DegenerateCase}).
\end{observation}

We are now able to give a slightly more compact formulation of Theorem \ref{thm:LSVariant}:

\begin{theorem}\label{thm:LSVariantReformulated}
	Let
	\[
		f(x):=(x-r)\prod_{j=2}^n(x-x_j)=x^n+a_1x^{n-1}+a_2x^{n-2}+\cdots+a_n
	\]
	be a real polynomial, where $r>0$ and $x_2,x_3,\ldots,x_n$ are nonzero complex numbers such that $\mathrm{Re}\,(x_j)\leq0$ for all $j\in\{2,3,\ldots,n\}$ and $\mathrm{Re}\,(x_j)<0$ for some $j\in\{2,3,\ldots,n\}$. Then, assuming $\prod_{j=2}^n(x-x_j)$ is not of the form (\ref{eq:DegenerateCase}), for each $k\in\{1,2,\ldots,n-2\}$, $a_k\leq0$ implies $a_{k+2}<0$.
\end{theorem}
\begin{proof}
	Let us write $f(x)=(x-r)g(x)$, where
	\[
		g(x):=\prod_{j=2}^n(x-x_j)=x^{n-1}+b_1x^{n-2}+b_2x^{n-3}+\cdots+b_{n-1}
	\]
	and let us define $b_0:=1$. Since $a_n=-rb_{n-1}<0$, we need only consider $k\leq n-3$.
	
	Suppose (to the contrary) that there exists $k\in\{1,2,\ldots,n-3\}$ such that
	\begin{equation}\label{eq:Newton1}
		a_k=b_k-rb_{k-1}\leq0
	\end{equation}
	and
	\begin{equation}\label{eq:Newton2}
		a_{k+2}=b_{k+2}-rb_{k+1}\geq0.
	\end{equation}
	Combining (\ref{eq:Newton1}) and (\ref{eq:Newton2}) gives
	\[
		b_kb_{k+1}\leq b_{k-1}b_{k+2},
	\]
	and so, by Theorem \ref{thm:NewtonLikeInequalities},
	\[
		b_kb_{k+1}=b_{k-1}b_{k+2},
	\]
	 which, by Observation \ref{obv:EqualityCase}, contradicts the hypotheses of the theorem.
\end{proof}

We will illustrate Theorems \ref{thm:Obreschkoff} and \ref{thm:LSVariantReformulated} with an example:

\begin{example}
	Consider the polynomial
	\[
		f(x):=(x-r)\left( (x+1)^2+\beta^2 \right)^m=x^{2m+1}+a_1x^{2m}+a_2x^{2m-1}+\cdots+a_{2m+1},
	\]
	where $r,\beta>0$. We note that $a_{2m+1}=-r(1+\beta^2)^m<0$ and so $f$ must have an odd number of sign changes. If $\beta\leq\sqrt{3}$, then by Theorem \ref{thm:Obreschkoff}, $f$ must have precisely one sign change. For larger values of $\beta$, we will see that $f$ may have many changes, but by Theorem \ref{thm:LSVariantReformulated}, the sequences
	\[
		\mathcal{T}_e:=(1,a_2,a_4,\ldots,a_{2m})
	\]
	and
	\[
		\mathcal{T}_o:=(a_1,a_3,\ldots,a_{2m+1})
	\]
	must each exhibit at most one sign change.
	
	If $\beta=\sqrt{2m+1}$ and $r=1+1/m$, it is not difficult to calculate that $a_{2m-1}=a_{2m}=0$, and in this case, Theorem \ref{thm:LSVariantReformulated} implies $a_k>0$: $k=1,2,\ldots,2m-2$, i.e. $f$ has precisely one sign change. Keeping this value of $\beta$ fixed, we may vary the location of the sign change by increasing $r$. In particular, with $r=2m$, we have $a_1=a_2=0$. We note that, with this value of $\beta$, the complex roots of $f$ lie outside of the wedge (\ref{eq:Sr3}), illustrating the well-known fact that location in this wedge is sufficient, but not necessary, for the coefficients of the polynomial to exhibit one sign change. The fact that two adjacent coefficients of $f$ can vanish simultaneously as $r$ is varied indicates that this value of $\beta$ is, in a sense, ``critical'': if $\beta$ were increased slightly beyond $\sqrt{2m+1}$, it would be possible to find a value of $r$ such that $f$ has three sign changes.
	
	Finally, let us consider an extreme case: if $\beta=2m$ and
	\[
		2m<r<2m+\frac{1}{2m},
	\]
	it is not difficult to check that $a_1<0$ and $a_{2m}>0$ (and hence $f$ has the maximal possible number of sign changes).
\end{example}

%\begin{example}
%	Consider the polynomial
%	\begin{align*}
%		f(x)&=x^{12}+a_1x^{11}+a_2x^{10}+\cdots+a_{12} \\
%		&=x^{12}+x^{11}+9 x^{10}-x^9+14 x^8-56 x^7-64 x^6-144 x^5-160 x^4
%	\end{align*}
%	with roots $(2,-1+2i,-1-2i,2i,-2i,2i,-2i,-1,0,0,0,0)$. The sequences
%	\[
%		(a_{2j})_{j=1}^6=(9,14,-64,-160,0,0)
%	\]
%	and
%	\[
%		(a_{2j-1})_{j=1}^6=(1,-1,-56,-144,0,0)
%	\]
%	correspond to $(s,t)=(3,4)$ and $(s',t')=(2,4)$ in Theorem \ref{thm:LSVariant}, respectively.
%\end{example}

%\section*{References}

%    Bibliographies can be prepared with BibTeX using amsplain,
%    amsalpha, or (for "historical" overviews) natbib style.
\bibliographystyle{amsplain}
\bibliography{Bibliography}

\providecommand{\bysame}{\leavevmode\hbox to3em{\hrulefill}\thinspace}
\providecommand{\MR}{\relax\ifhmode\unskip\space\fi MR }
% \MRhref is called by the amsart/book/proc definition of \MR.
\providecommand{\MRhref}[2]{%
  \href{http://www.ams.org/mathscinet-getitem?mr=#1}{#2}
}
\providecommand{\href}[2]{#2}
\begin{thebibliography}{10}

\bibitem{AlesinaGaluzzi2000}
Alberto~Claudio Alesina and Massimo Galuzzi, \emph{VincentÕs theorem from a
  modern point of view}, Rendiconti del Circolo Matematico di Palermo Serie II
  \textbf{64} (2000), 179--191.

\bibitem{Bhattacharyya1987}
S.~P. Bhattacharyya, \emph{Robust stabilization against structured
  perturbations}, Lecture Notes in Control and Information Sciences, vol.~99,
  Springer, Berlin, 1987.

\bibitem{Biehler1879}
M.~Biehler, \emph{Sur une classe d'\'{e}quations alg\'{e}briques dont toutes
  les racines sont r\'{e}elles}, J. Reine Angew. Math \textbf{87} (1879),
  350--352.

\bibitem{NewtonLikeInequalities}
Richard Ellard and Helena \v{S}migoc, \emph{Families of {N}ewton-like
  inequalities for sets of self-conjugate complex numbers}, arXiv:1604.05148
  (2016).

\bibitem{Gantmacher}
Felix~R. Gantmacher, \emph{The theory of matrices, vol. 2}, AMS Chelsea
  Publishing, 2000, (Translated from the Russian by K.A. Hirsch).

\bibitem{Hermite1856}
C.~Hermite, \emph{Sur les nombre des racines d'une \'{e}quation alg\'{e}brique
  comprise entre des limites donn\'{e}es}, J. Reine Angew. Math \textbf{52}
  (1856), 39--51.

\bibitem{Holtz2003}
Olga Holtz, \emph{{H}ermite-{B}iehler, {R}outh-{H}urwitz, and total
  positivity}, Linear Algebra and its Applications \textbf{372} (2003),
  105--110.

\bibitem{Kharitonov1978}
V.~L. Kharitonov, \emph{Asymptotic stability of an equilibrium position of a
  family of systems of linear differential equations}, Differential'nye
  Uravneniya \textbf{14} (1978), 2086--2088.

\bibitem{LaffeySmigoc}
Thomas~J. Laffey and Helena \v{S}migoc, \emph{Nonnegative realization of
  spectra having negative real parts}, Linear Algebra and its Applications
  \textbf{416} (2006), no.~1, 148 -- 159.

\bibitem{Obreschkoff1966}
N.~Obreschkoff, \emph{Verteilung und berechnung der nullstellen reeller
  polynome}, Zeitschrift f\"{u}r Angewandte Mathematik und Mechanik \textbf{46}
  (1966), no.~7, 475--475.

\bibitem{Perfect1953}
H.~Perfect, \emph{Methods of constructing certain stochastic matrices}, Duke
  Math. J. \textbf{20} (1953), 395--404.

\bibitem{Suleimanova1949}
H.R. Sule\v{i}manova, \emph{Stochastic matrices with real characteristic
  values}, Dokl. Akad. Nauk. S.S.S.R. \textbf{66} (1949), 343--345, (In
  Russian).

\bibitem{Vincent1836}
Alexandre Joseph~Hidulphe Vincent, \emph{Sur la r\'{e}solution des
  \'{e}quations num\'{e}riques}, Journal de Math\'{e}matiques Pures et
  Appliqu\'{e}es \textbf{1} (1836), 341--372.

\end{thebibliography}

\end{document}